\documentclass[11pt]{amsart}

\usepackage{amssymb,amsmath}
\usepackage{mathrsfs}

\large\normalsize
\newtheorem{theorem}{Theorem}
\newtheorem{lemma}[theorem]{Lemma}

\newtheorem{definition}[theorem]{Definition}
\newtheorem{example}[theorem]{Example}
\newtheorem{remark}[theorem]{Remark}

\newcommand{\R}{\ensuremath{\mathbb{R}}}
\newcommand{\Z}{\ensuremath{\mathbb{Z}}}
\newcommand{\C}{\ensuremath{\mathbb{C}}}

\newcommand{\D}{\ensuremath{\mathbb{D}}}
\newcommand{\T}{\ensuremath{\mathbb{T}}}

\DeclareMathOperator{\diag}{diag}
\DeclareMathOperator{\rank}{rank}
\DeclareMathOperator{\im}{Im}

\numberwithin{equation}{section}
\numberwithin{theorem}{section}

\small\normalsize

\begin{document}

\title[Finite Time-Varying Spectral Theory]{Spectral Theory for Second-Order Vector Equations on Finite Time-Varying Domains}

\author[anderson]{Douglas R. Anderson} 
\address{Department of Mathematics and Computer Science, Concordia College, Moorhead, MN 56562 United States\\visiting the School of Mathematics, The University of New South Wales Sydney 2052, Australia}
\email{andersod@cord.edu}


\begin{abstract} 
In this study, we are concerned with spectral problems of second-order vector dynamic equations with two-point boundary value conditions and mixed derivatives, where the matrix-valued coefficient of the leading term may be singular, and the domain is non-uniform but finite. A concept of self-adjointness of the boundary conditions is introduced. The self-adjointness of the corresponding dynamic operator is discussed on a suitable admissible function space, and fundamental spectral results are obtained. The dual orthogonality of eigenfunctions is shown in a special case. Extensions to even-order Sturm-Liouville dynamic equations, linear Hamiltonian and symplectic nabla systems on general time scales are also discussed.
\end{abstract}

\maketitle\thispagestyle{empty}


\section{Introduction}\label{secintro}

Shi and Chen \cite{sc1} provide an analysis for second-order vector difference equations of the form
$$ -\nabla\left(C_n\Delta x_n\right)+B_nx_n=\lambda \omega_nx_n, \quad n\in[1,N]\cap\Z, \quad N\ge 2, $$
where the forward difference operator $\Delta x_n:=x_{n+1}-x_n$ and the backward difference operator $\nabla x_n:=x_n-x_{n-1}$ are utilized. This is one vector form related to the well-known self-adjoint scalar form, see Kelley and Peterson \cite[Chapters $6-7$]{kp},
$$ -\Delta\left(c_{n-1}\Delta x_{n-1}\right)+b_nx_{n}=\lambda\omega_n x_n. $$
Recently, however, there has been a growing interest in time-varying systems such as by Davis, Gravagne, Jackson, and Marks \cite{baylor} or Gravagne, Davis, and DaCunha \cite{baylor2}, that is discrete systems with varying step-size between domain points; these systems are sometimes also called isolated time scales. An important example in this class of problems would be quantum ($q$-difference) equations; see Simmons \cite[Chapter B.5]{simmons} and Kac and Cheung \cite{kac} for quantum calculus, and Erbe and Hilger \cite{eh} and Bohner and Peterson \cite{bp1} for more on time scales, including isolated time scales. 

Motivated by the appeal of time-varying domains and the ability to simultaneously unify and generalize recent and classic results, we introduce here a finite-dimensional analysis of second-order vector dynamic equations of the form
\begin{equation}\label{introeq}
 -(Px^\Delta)^\nabla(t)+Q(t)x(t)=\lambda \omega(t)x(t), \quad t\in[a,b]_\T, \quad b\ge \sigma(a),
\end{equation}
with boundary conditions
\begin{equation}\label{qbc}
 R\left(\begin{smallmatrix} -x^\rho(a) \\ x(b) \end{smallmatrix}\right)+S\left(\begin{smallmatrix} P^\rho(a)x^{\nabla}(a) \\ P(b)x^{\Delta}(b) \end{smallmatrix}\right)=0.
\end{equation}
A solution $x$ of \eqref{introeq} is defined on $[\rho(a),\sigma(b)]_\T$. Here $\T$ is a (finite) isolated time scale, the discrete interval is given by
$$ [a,b]_\T:=\{a,\sigma(a),\sigma^2(a),\cdots,\rho(b),b\}, \quad b=\sigma^{N-1}(a), $$ 
the $d\times d$ matrices $P(t)$ for $t\in[\rho(a),b]_\T$, $Q(t)$ and $\omega(t)$ for $t\in[a,b]_\T$, are all Hermitian with $P^\rho(a)$ and $P(b)$ invertible, $\omega(t)>0$ (positive definite) for $t\in[a,b]_\T$, and $R$ and $S$ in \eqref{qbc} are $2d\times 2d$ matrices with $\rank(R,S)=2d$. The differential operators in \eqref{introeq} are given, respectively, by 
$$ x^\Delta(t) = \frac{x^\sigma(t)-x(t)}{\mu_\sigma(t)} \quad\text{and}\quad x^\nabla(t) = \frac{x(t)-x^\rho(t)}{\mu_\rho(t)}=x^{\Delta\rho}(t), $$
where time-varying step-sizes are given by $\mu_\sigma(t):=\sigma(t)-t>0$, $\mu_\rho(t):=t-\rho(t)>0$, with $\sigma(t)$ the immediate domain point to the right of $t$, $\rho(t)$ the immediate domain point to the left of $t$, and where compositions are often denoted $f(\alpha(t))=f^\alpha(t)$. Note that the coefficient matrix $P(t)$ in \eqref{introeq} may be singular for $t\in[a,\rho(b)]_\T$. Moreover, the second-order dynamic operator given in \eqref{introeq} may not be formally self-adjoint in the space of vector series of the form 
$$ \left\{x(\sigma^n(a))\right\}_{n=-1}^{N}, \quad \sigma^{-1}(a):=\rho(a), \quad \sigma^0(a):=a, \quad \sigma^n(a):=\sigma(\sigma^{n-1}(a)), $$ 
subject to the boundary conditions in \eqref{qbc} unless those conditions themselves are self-adjoint in some sense to be defined. On the other hand, even if the boundary conditions \eqref{qbc} are self-adjoint, the second-order dynamic operator may still not be self-adjoint. As in the difference equations case \cite{sc1}, self-adjointness depends on whether the vectors at the boundaries, namely $x^\rho(a)$ and $x^\sigma(b)$, can be solved from \eqref{qbc} in terms of $x(a)$ and $x(b)$. Indeed if this is the case, we will call \eqref{qbc} proper self-adjoint boundary conditions. If \eqref{qbc} happen to be improper, we can find a smaller space where the second-order dynamic operator is self-adjoint. Once this is sorted out, we will give the appropriate analysis of the eigenvalue problem for \eqref{introeq} and \eqref{qbc}.

The particular appeal of \eqref{introeq} is that it is still a discrete problem, but with non-constant step-size between
domain points. As
$$ \lim_{\mu_\sigma\rightarrow 0}x^\Delta(t) = \lim_{\mu_\rho\rightarrow 0}x^\nabla(t)=x'(t), \quad '=\frac{d}{dt}, $$
for differentiable functions $x$, these dynamic results serve as an alternative discrete analog to the differential equations case. There are many papers on the spectral theory for difference equations, including recent papers such as Ji and Yang \cite{jy1,jy2}, Shi \cite{shi1,shi2}, Shi and Chen \cite{sc1,sc2}, Shi and Wu \cite{sw}, Shi and Yan \cite{sy}, Sun and Shi \cite{ss}, Sun, Shi, and Chen \cite{ssc}, Wang and Shi \cite{ws}, and Bohner,  Do\v{s}l\'{y}, and Kratz \cite{bdk}, but none on dynamic vector equations. Thus this work will continue and extend the discussion of self-adjoint equations on time scales found in \cite{ah,agh,dr,gus,messer}. As in the uniformly-discrete case \cite{sc1}, we use mixed derivatives, as $-\nabla$ is a natural adjoint of $\Delta$. To emphasis the dimensional analysis we focus on isolated time scales, but a similar treatment can be made on general time scales.

The analysis of \eqref{introeq} and its solutions will unfold as follows, largely motivated by Shi and Chen \cite{sc1}. In Section 2 we will give a definition of the self-adjointness of the boundary conditions in \eqref{qbc}. Section 3 contains the introduction of a suitable admissible function space in which the corresponding dynamic operator is self-adjoint, followed in Section 4 by some fundamental spectral results. In the event of proper self-adjoint boundary conditions, the dual orthogonality of eigenfunctions will be given. In Section 5 we discuss the possibility of extending these results to even-order Sturm-Liouville equations, linear Hamiltonian and symplectic nabla systems on time scales.


\section{Self-Adjoint Boundary Conditions}\label{secselfadj}

In this section we are concerned with the second-order vector dynamic operator $\ell$ given by
\begin{equation}\label{saeq}
  \ell x(t)=\omega^{-1}(t)\left[-(Px^\Delta)^\nabla(t)+Q(t)x(t)\right], \quad t\in[a,b]_\T,
\end{equation}
where we will write $x\in\mathscr{R}$ if $\left\{x(\sigma^n(a))\right\}_{n=-1}^{N}$ satisfies the boundary conditions given in \eqref{qbc}. If we set
$$ \ell[\rho(a),\sigma(b)]:=\left\{x=\left\{x(\sigma^n(a))\right\}_{n=-1}^{N}:x(t)\in\C^d, t\in[\rho(a),\sigma(b)])_\T\right\}, $$
then $\dim \ell[\rho(a),\sigma(b)]=d(N+2)$. Define the (weighted) inner product to be given by
\begin{equation}\label{ip}
 \big\langle x,y \big\rangle = \sum_{t\in[a,b]_\T} y^*(t)\omega(t)x(t)\mu_\rho(t) \quad\text{for}\quad x,y\in \ell[\rho(a),\sigma(b)],
\end{equation}
where $\omega(t)$ is as in \eqref{introeq}, $\mu_\rho(t)=t-\rho$ is the left-graininess at $t$, and $y^*(t)$ denotes the complex conjugate transpose of $y(t)$. Using standard notation, $x\perp y$ will mean $\langle x,y \rangle=0$. 


\begin{theorem}[Lagrange Identity]\label{thm23}
Assume the $d\times d$ matrices $P(t)$ for $t\in[\rho(a),b]_\T$, $Q(t)$ and $\omega(t)$ for $t\in[a,b]_\T$, are all Hermitian with $P^\rho(a)$ and $P(b)$ invertible, and $\omega(t)>0$ for $t\in[a,b]_\T$.  Then for $x,y\in \ell[\rho(a),\sigma(b)]$ we have that
\begin{equation}\label{eq23}
 \big\langle \ell  x,y \big\rangle - \big\langle x,\ell  y \big\rangle = \left[(Py^\Delta)^*x-y^*Px^\Delta\right](b) -\left[(Py^\Delta)^*x-y^*Px^\Delta\right]^\rho(a),
\end{equation}
where the inner product is given in \eqref{ip}.
\end{theorem}

\begin{proof}
Let $x,y\in \ell[\rho(a),\sigma(b)]$. Using summation by parts \cite[Theorem 2.7]{ag} we have
\begin{eqnarray*}
\langle \ell  x,y \rangle-\langle x,\ell  y \rangle 
 &=& \sum_{t\in[a,b]_\T} y^*(t)\omega(t)\ell  x(t)\mu_\rho(t) - \sum_{t\in[a,b]_\T} (\ell y)^*(t)\omega(t)x(t)\mu_\rho(t) \\
 &=& -\sum_{t\in[a,b]_\T} y^*(t)\left((Px^\Delta)^\nabla-Qx\right)(t)\mu_\rho(t) \\
 & & + \sum_{t\in[a,b]_\T} \left((Py^\Delta)^\nabla-Qy\right)^*(t)x(t)\mu_\rho(t) \\
 &=& -\sum_{t\in[a,b]_\T} y^*(t)(Px^\Delta)^\nabla(t)\mu_\rho(t) + \sum_{t\in[a,b]_\T} (Py^\Delta)^{\nabla*}(t)x(t)\mu_\rho(t) \\ 
 &=& -\left[y^*Px^\Delta\right](t)\Big|^{b}_{\rho(a)} + \sum_{t\in[a,b]_\T} (Py^\Delta)^{*\rho}(t)x^\nabla(t)\mu_\rho(t) \\ 
 & & + \sum_{t\in[a,b]_\T} (Py^{\Delta})^{\nabla*}(t)x(t)\mu_\rho(t) \\
 &=& -\left[y^*Px^\Delta\right](t)\Big|^{b}_{\rho(a)} + \left[(Py^\Delta)^*x\right](t)\Big|_{\rho(a)}^{b} \\
 &=& \left[(Py^\Delta)^*x-y^*Px^\Delta\right](t)\Big|^{b}_{\rho(a)}
\end{eqnarray*}
and the result follows.
\end{proof}

The above theorem, Theorem $\ref{thm23}$, is also valid on general time scales; see Anderson and Buchholz \cite[Theorem 2.13]{ab}.
\begin{definition}\label{sabc}
The boundary conditions given in \eqref{qbc} are called self-adjoint iff
\begin{equation}\label{kr216}
 \left[(Py^\Delta)^*x-y^*Px^\Delta\right](t)\Big|^{b}_{\rho(a)} = 0
\end{equation}
for $x,y\in\mathscr{R}$.
\end{definition}

The next theorem then follows easily from the definition and Theorem \ref{thm23}.


\begin{theorem}\label{thm25}
Assume the $d\times d$ matrices $P(t)$ for $t\in[\rho(a),b]_\T$, $Q(t)$ and $\omega(t)$ for $t\in[a,b]_\T$, are all Hermitian with $P^\rho(a)$ and $P(b)$ invertible, $\omega(t)>0$ for $t\in[a,b]_\T$, and $R$ and $S$ in \eqref{qbc} are $2d\times 2d$ matrices with $\rank(R,S)=2d$. If the boundary conditions in \eqref{qbc} are self-adjoint, then
$$ \langle \ell  x,y \rangle=\langle x,\ell  y \rangle $$
for $x,y\in \ell[\rho(a),\sigma(b)]$ with $x,y\in\mathscr{R}$.
\end{theorem}

\begin{lemma}
Assume that $R$ and $S$ are $2d\times 2d$ matrices with $\rank(R,S)=2d$. Then 
\begin{equation}
 SR^*=RS^*
\end{equation}
if and only if the boundary conditions \eqref{qbc} are self-adjoint.
\end{lemma}

\begin{proof}
We follow Kratz \cite[Proposition 2.1.1]{kratz}. Let $x\in \ell[\rho(a),\sigma(b)]$ with $x\in\mathscr{R}$. By the definition of $\mathscr{R}$, $x$  satisfies the boundary conditions \eqref{qbc}. Let $\mathcal{K}=\left(\begin{smallmatrix} B \\ C \end{smallmatrix}\right)$ be any $4d\times 2d$ matrix with $\im \mathcal{K}=\ker(B,C)$, so that $SB+RC=0$, $\rank\mathcal{K}=2d$, and so that $x\in\mathscr{R}$ if and only if $\left(\begin{smallmatrix} x^\rho(a) \\ -x(b) \end{smallmatrix}\right)=B\eta$, $\left(\begin{smallmatrix} P^\rho(a)x^\nabla(a) \\ P(b)x^{\Delta}(b) \end{smallmatrix}\right)=C\eta$ for some $\eta\in\C^{2d}$. This results in \eqref{kr216} holding for all $x,y\in\mathscr{R}$ if and only if $\eta_1^*[B^*C-C^*B]\eta_2=0$ for all $\eta_1,\eta_2\in\C^{2d}$, that is to say $B^*C=C^*B$.

First, assume that $RS^*=SR^*$. Since $\rank(R,S)=2d$, we see that $\im\left(\begin{smallmatrix} S^* \\ -R^* \end{smallmatrix}\right)=\ker(R,S)$. Consequently, the matrix $\mathcal{K}$ given above can be taken to be $\left(\begin{smallmatrix} S^* \\ -R^* \end{smallmatrix}\right)$; by the Hermitian assumption on $RS^*$ we see that $B^*C$ is Hermitian, and thus the boundary conditions \eqref{qbc} are self-adjoint by Definition \ref{sabc}. The rest of the proof is identical to \cite[Proposition 2.1.1]{kratz} and is omitted.
\end{proof}

\begin{lemma}\label{lem25}
Assume the $d\times d$ matrices $P(t)$ for $t\in[\rho(a),b]_\T$, $Q(t)$ and $\omega(t)$ for $t\in[a,b]_\T$, are all Hermitian with $P^\rho(a)$ and $P(b)$ invertible, $\omega(t)>0$ for $t\in[a,b]_\T$, and $R$ and $S$ in \eqref{qbc} are $2d\times 2d$ matrices with $\rank(R,S)=2d$. If the boundary conditions in \eqref{qbc} are self-adjoint, then $x\in\mathscr{R}$ if and only if there exists a unique vector $\eta\in\C^{2d}$ such that
\begin{equation}\label{sc28}
 \left(\begin{smallmatrix} -x^\rho(a) \\ x(b) \end{smallmatrix}\right) = -S^*\eta, \qquad \left(\begin{smallmatrix} P^\rho(a)x^\nabla(a) \\ P(b)x^{\Delta}(b) \end{smallmatrix}\right) = R^*\eta.
\end{equation}
\end{lemma}

\begin{proof}
See Kratz \cite[Proposition 2.1.2]{kratz} and Shi and Chen \cite[Lemma 2.2]{sc1}.
\end{proof}


\section{Admissible Function Space and Self-Adjointness of the Dynamic Operator}

Throughout this section we assume the $d\times d$ matrices $P(t)$ for $t\in[\rho(a),b]_\T$, $Q(t)$ and $\omega(t)$ for $t\in[a,b]_\T$, are all Hermitian with $P^\rho(a)$ and $P(b)$ invertible, $\omega(t)>0$ for $t\in[a,b]_\T$, the matrices $R$ and $S$ in \eqref{qbc} are $2d\times 2d$ matrices with $\rank(R,S)=2d$, and the boundary conditions in \eqref{qbc} are self-adjoint. We begin the construction of an admissible function space to establish the self-adjointness of a dynamic operator related to \eqref{introeq}, following the development in the uniformly discrete case presented in \cite{sc1}.

Let $R=(R_1,R_2)$ and $S=(S_1,S_2)$, where $R_k$ and $S_k$ are $2d\times d$ matrices, $k=1,2$. By Lemma \ref{lem25}, $x\in\mathscr{R}$ if
and only if there exists a unique vector $\eta\in\C^{2d}$ such that
\begin{eqnarray}
 x^\rho(a)=S_1^*\eta, \quad x(a)=(S_1^*+\mu_\rho(a)P^{\rho-1}(a)R_1^*)\eta, \nonumber \\
 x(b)=-S_2^*\eta, \quad x^\sigma(b)=(-S_2^*+\mu_\sigma(b)P^{-1}(b)R_2^*)\eta, \label{sc31}
\end{eqnarray}
where $P^{\rho-1}=(P^\rho)^{-1}$. Rewrite the boundary conditions \eqref{qbc} as
\begin{equation}\label{sc32}
 \Gamma\left(\begin{smallmatrix} P^\rho(a) & 0 \\ 0 & -P(b)\end{smallmatrix}\right) 
 \left(\begin{smallmatrix} x^\rho(a) \\ x^\sigma(b) \end{smallmatrix}\right) 
 = \left(\frac{1}{\mu_\rho(a)}S_1P^\rho(a), R_2-\frac{1}{\mu_\sigma(b)}S_2P(b)\right)\left(\begin{smallmatrix} x(a) \\ x(b) \end{smallmatrix}\right),
\end{equation}
where the coefficient matrix $\Gamma$ is given by
$$ \Gamma=\left(R_1P^{\rho-1}(a)+\frac{1}{\mu_\rho(a)}S_1,\frac{1}{\mu_\sigma(b)}S_2\right). $$
Set 
$$ r=\rank \left(R_1+\frac{1}{\mu_\rho(a)}S_1P^{\rho}(a),\frac{1}{\mu_\sigma(b)}S_2\right)\in\{0,1,\cdots,2d\}, $$
where this value $r$ will play a key role in the development to follow. If the matrix $\Gamma$ is invertible, in other words if $r=2d$, then the boundary conditions \eqref{qbc} are proper, since \eqref{sc32} is solvable for $x^\rho(a)$ and $x^\sigma(b)$, respectively. We then see that, as $x^\rho(a)$ and $x^\sigma(b)$ are not weighted with respect to the weight function $\omega$, the $dN$-dimensional space $ L[\rho(a),\sigma(b)]:=\left\{x\in \ell[\rho(a),\sigma(b)]:x\in\mathscr{R}\right\}$ might serve as a suitable admissible space. If $r<2d$, however, then from \eqref{sc32} we have that the $2d$ components of $x(a)$ and $x(b)$ are themselves linked by $2d-r$ relations. From \eqref{introeq} we see that $x(a)$ and $x(b)$ are weighted by $\omega(a)$ and $\omega(b)$ in the first and last vector equations. In the two vector equations, via some transformation, only $r$ scalar equations are really weighted, while the remaining $2d-r$ ones involve $x^\rho(a)$, $x(a)$, $x^\sigma(a)$, $x^\rho(b)$, $x(b)$, and $x^\sigma(b)$ but not $\lambda$; consequently they can be viewed as extra conditions for the admissible functions. We will now show this as follows. 
 
Using standard matrix theory \cite{hj}, there exist $2d\times 2d$-unitary matrices $U$ and $V$ such that
\begin{equation}\label{sc33}
 U^*\Gamma V = \diag\{0,M\},
\end{equation}
where $M$ is an $r\times r$ matrix with $\rank M=r$. Let $U=(U_1,U_2)$ and $V=(V_1,V_2)$, where $U_1$ and $V_1$ are $2d\times(2d-r)$ matrices, $U_2$ and $V_2$ are $2d\times r$ matrices. Since $V$ is unitary, from \eqref{sc33} we have that
\begin{equation}\label{sc34}
 V_1^*V_1=I_{2d-r}, \quad V_1^*V_2=0_{(2d-r)\times r}, \quad V_2^*V_2=I_r,
\end{equation}
\begin{equation}\label{sc35}
 U_1^*\Gamma=0_{(2d-r)\times 2d}, \quad V_2=\Gamma^*U_2M^{*-1}.
\end{equation}
Using \eqref{sc31} and \eqref{sc33}, we see that
\begin{equation}\label{sc36}
 \left(\begin{smallmatrix} x(a) \\ -x(b) \end{smallmatrix}\right)=V\diag\{0,M^*\}U^*\eta=(0_{2d\times(2d-r)},V_2M^*)U^*\eta,
\end{equation}
so that by \eqref{sc34} we have
\begin{equation}\label{sc37}
 V_1^*\left(x^T(a), -x^T(b)\right)=0, \qquad x^T=\text{transpose of $x$}.
\end{equation}
The first and last relations in \eqref{introeq} can be written as
$$ \lambda \begin{pmatrix} x(a) \\ -x(b)\end{pmatrix} 
= \begin{pmatrix} \omega^{-1}(a) & 0 \\ 0 & \omega^{-1}(b) \end{pmatrix} 
\begin{pmatrix} \frac{-1}{\left(\mu_\rho(a)\right)^2}P^\rho(a)x^\rho(a) + \tilde{P}(a)x(a)-\frac{1}{\mu_\rho(a)\mu_\sigma(a)}P(a)x^\sigma(a) \\ \frac{1}{\left(\mu_\rho(b)\right)^2}P^\rho(b)x^\rho(b) - \tilde{P}(b)x(b)+\frac{1}{\mu_\rho(b)\mu_\sigma(b)}P(b)x^\sigma(b) \end{pmatrix}, $$
where we have taken 
$$ \tilde{P}(t)=Q(t)+\frac{1}{\mu_\rho(t)\mu_\sigma(t)}P(t)+\frac{1}{\left(\mu_\rho(t)\right)^2}P^\rho(t). $$
From \eqref{sc37} we obtain the $2d-r$ scalar equations
\begin{equation}\label{sc38}
 V_1^*\begin{pmatrix} \omega^{-1}(a) & 0 \\ 0 & \omega^{-1}(b) \end{pmatrix} \begin{pmatrix} \frac{-1}{\left(\mu_\rho(a)\right)^2}P^\rho(a)x^\rho(a) + \tilde{P}(a)x(a)-\frac{1}{\mu_\rho(a)\mu_\sigma(a)}P(a)x^\sigma(a) \\ \frac{1}{\left(\mu_\rho(b)\right)^2}P^\rho(b)x^\rho(b) - \tilde{P}(b)x(b)+\frac{1}{\mu_\rho(b)\mu_\sigma(b)}P(b)x^\sigma(b) \end{pmatrix}=0.
\end{equation}
If $x$ satisfies \eqref{sc38} we will write $x\in\mathscr{A}$. Noting the absence of $\lambda$ from \eqref{sc38}, we can think of \eqref{sc38} as additional conditions together with the boundary conditions \eqref{qbc}. Thus, we define the admissible function space
$$ L^2_\omega[\rho(a),\sigma(b)]:=\{x\in\ell[\rho(a),\sigma(b)]: x\in\mathscr{R}\cap\mathscr{A}\}. $$
From \eqref{sc34} we see that $x\in\mathscr{A}$ if and only if $x$ satisfies
\begin{equation}\label{sc39}
  \begin{pmatrix} \frac{-1}{\left(\mu_\rho(a)\right)^2}P^\rho(a)x^\rho(a) + \tilde{P}(a)x(a)-\frac{1}{\mu_\rho(a)\mu_\sigma(a)}P(a)x^\sigma(a) \\ \frac{1}{\left(\mu_\rho(b)\right)^2}P^\rho(b)x^\rho(b) - \tilde{P}(b)x(b)+\frac{1}{\mu_\rho(b)\mu_\sigma(b)}P(b)x^\sigma(b) \end{pmatrix} = \begin{pmatrix} \omega(a) & 0 \\ 0 & \omega(b) \end{pmatrix}V_2 \gamma
\end{equation}
for some $\gamma\in\C^r$. As $P^\rho(a)$ and $P(b)$ are invertible, \eqref{sc39} implies that
\begin{eqnarray}\label{sc310}
\begin{pmatrix} x^\rho(a) \\ x^\sigma(b)\end{pmatrix} 
&=& \begin{pmatrix} P^{\rho-1}(a)\left(\left(\mu_\rho(a)\right)^2\tilde{P}(a)x(a)-\frac{\mu_\rho(a)}{\mu_\sigma(a)}P(a)x^\sigma(a)\right) \\ P^{-1}(b)\left(-\frac{\mu_\sigma(b)}{\mu_\rho(b)}P^\rho(b)x^\rho(b)+\mu_\rho(b)\mu_\sigma(b)\tilde{P}(b)x(b)\right) \end{pmatrix} \nonumber \\
& & + \begin{pmatrix}  -\left(\mu_\rho(a)\right)^2P^{\rho-1}(a)\omega(a) & 0 \\ 0 & \mu_\rho(b)\mu_\sigma(b)P^{-1}(b)\omega(b) \end{pmatrix}V_2 \gamma.
\end{eqnarray}
Let $x\in L^2_\omega[\rho(a),\sigma(b)]$. Inserting \eqref{sc310} into boundary conditions \eqref{qbc} or \eqref{sc32} yields
\begin{eqnarray}
 \Gamma \diag \{\left(\mu_\rho(a)\right)^2\omega(a), \mu_\rho(b)\mu_\sigma(b)\omega(b)\}V_2\gamma
  &=& \Gamma \begin{pmatrix} \left(\mu_\rho(a)\right)^2\tilde{P}(a)x(a)-\frac{\mu_\rho(a)}{\mu_\sigma(a)}P(a)x^\sigma(a) \\ \frac{\mu_\sigma(b)}{\mu_\rho(b)}P^\rho(b)x^\rho(b)-\mu_\rho(b)\mu_\sigma(b)\tilde{P}(b)x(b) \end{pmatrix} \nonumber \\
 & & -\left(\frac{1}{\mu_\rho(a)}S_1P^\rho(a), R_2-\frac{1}{\mu_\sigma(b)}S_2P(b)\right)\left(\begin{smallmatrix} x(a) \\ x(b) \end{smallmatrix}\right) \\
 & & =:f\left(x(a),x^\sigma(a),x^\rho(b),x(b)\right). \label{sc311}
\end{eqnarray}
From \eqref{sc33} we have
\begin{equation}\label{sc312}
 \Gamma  \diag \{\left(\mu_\rho(a)\right)^2\omega(a), \mu_\rho(b)\mu_\sigma(b)\omega(b)\}V_2=U\left(0_{r\times(2d-r)},J^T\right)^T,
\end{equation}
where $J=MV_2^*\diag\{\left(\mu_\rho(a)\right)^2\omega(a), \mu_\rho(b)\mu_\sigma(b)\omega(b)\}V_2$ is an $r\times r$ invertible matrix, as $\rank M=\rank V_2=r$, and $\omega(a)$ and $\omega(b)$ are positive definite by assumption. Multiplying \eqref{sc311} on the left by $U^*$, we have from \eqref{sc35} (in particular $U_1^*S_2=0$) and \eqref{sc312} that
\begin{equation}\label{sc313}
 U_1^*\left(\frac{1}{\mu_\rho(a)}S_1P^\rho(a)x(a)+ R_2x(b)\right) = 0
\end{equation}
\begin{equation}\label{sc314}
 \gamma=J^{-1}U_2^*f\left(x(a),x^\sigma(a),x^\rho(b),x(b)\right).
\end{equation}
We are now in a position to get a useful characterization for a function $x$ to be a member of $L^2_\omega[\rho(a),\sigma(b)]$. As we have the necessary structure in place and in parallel with the discrete case presented by \cite{sc1}, we omit the proofs of the following key results; for more details, please see \cite{sc1}.

\begin{lemma}
We have $x\in L^2_\omega[\rho(a),\sigma(b)]$ if and only if $x$ satisfies \eqref{sc310} and \eqref{sc313} in which $\gamma$ is determined by \eqref{sc314}.
\end{lemma}

\begin{lemma}
We have $m:=\dim L^2_\omega[\rho(a),\sigma(b)]=d(N-2)+r$.
\end{lemma}

\begin{theorem}
Assume the $d\times d$ matrices $P(t)$ for $t\in[\rho(a),b]_\T$, $Q(t)$ and $\omega(t)$ for $t\in[a,b]_\T$, are all Hermitian with $P^\rho(a)$ and $P(b)$ invertible, $\omega(t)>0$ for $t\in[a,b]_\T$, and $R$ and $S$ in \eqref{qbc} are $2d\times 2d$ matrices with $\rank(R,S)=2d$. If the boundary conditions in \eqref{qbc} are self-adjoint, then $L^2_\omega[\rho(a),\sigma(b)]$ is an $m$-dimensional Hilbert space with the inner product defined in \eqref{ip}.
\end{theorem}

Let
$$ L:=\ell\big|_{L_{\omega}^2[\rho(a),\sigma(b)]}, \quad\text{where $\ell$ is given in \eqref{saeq}.} $$
We will see the self-adjointness of this dynamic operator.

\begin{lemma}
The dynamic operator $L$ maps $L_{\omega}^2[\rho(a),\sigma(b)]$ into itself.
\end{lemma}

\begin{proof}
See \cite[Proposition 3.3]{sc1}, which uses \eqref{sc310} with $x$ replaced by $Lx$ as in

\begin{eqnarray}\label{sc316}
\begin{pmatrix} Lx^\rho(a) \\ Lx^\sigma(b)\end{pmatrix} 
&=& \begin{pmatrix} P^{\rho-1}(a)\left(\left(\mu_\rho(a)\right)^2\tilde{P}(a)Lx(a)-\frac{\mu_\rho(a)}{\mu_\sigma(a)}P(a)Lx^\sigma(a)\right) \\ P^{-1}(b)\left(-\frac{\mu_\sigma(b)}{\mu_\rho(b)}P^\rho(b)Lx^\rho(b)+\mu_\rho(b)\mu_\sigma(b)\tilde{P}(b)Lx(b)\right) \end{pmatrix} \nonumber \\
& & + \begin{pmatrix}  -\left(\mu_\rho(a)\right)^2P^{\rho-1}(a)\omega(a) & 0 \\ 0 & \mu_\rho(b)\mu_\sigma(b)P^{-1}(b)\omega(b) \end{pmatrix}V_2 \hat{\gamma},
\end{eqnarray}
where $\hat{\gamma}$ is determined by \eqref{sc314} with $x(a)$, $x^\sigma(a)$, $x^\rho(b)$, and $x(b)$ replaced by $Lx(a)$, $Lx^\sigma(a)$, $Lx^\rho(b)$, and $Lx(b)$.
\end{proof}

\begin{theorem}\label{scthm32}
Assume the $d\times d$ matrices $P(t)$ for $t\in[\rho(a),b]_\T$, $Q(t)$ and $\omega(t)$ for $t\in[a,b]_\T$, are all Hermitian with $P^\rho(a)$ and $P(b)$ invertible, $\omega(t)>0$ for $t\in[a,b]_\T$, and $R$ and $S$ in \eqref{qbc} are $2d\times 2d$ matrices with $\rank(R,S)=2d$. If the boundary conditions in \eqref{qbc} are self-adjoint, then the dynamic operator $L$ is self-adjoint.
\end{theorem}


\section{Dynamic Spectral Theory}

As in the previous section, the machinery for a dynamic spectral theory is in place. All proofs of the following results are modifications of those found in the discrete case in \cite{sc1}, and thus are omitted here. For completeness, however, we list the main ideas and theorems that hold in this setting as well.

\begin{definition}
A complex number $\lambda\in\C$ is an eigenvalue of \eqref{introeq}, \eqref{qbc} if and only if there exists $x\in \ell[\rho(a),\sigma(b)]$ that is nonzero and solves the boundary value problem \eqref{introeq}, \eqref{qbc}. The nonzero solution $x$ is a corresponding eigenfunction for $\lambda$, denoted $x(\cdot,\lambda)$.
\end{definition}

Following the discussion in the previous section, any eigenfunction of \eqref{introeq}, \eqref{qbc} is in the space $L_{\omega}^2[\rho(a),\sigma(b)]$. Therefore, the following fundamental spectral results for \eqref{introeq}, \eqref{qbc} can be directly concluded by Theorem \ref{scthm32} and by
employing the spectral theory of self-adjoint linear operators in Hilbert spaces.
spaces.

\begin{theorem}\label{scthm41}
Assume the $d\times d$ matrices $P(t)$ for $t\in[\rho(a),b]_\T$, $Q(t)$ and $\omega(t)$ for $t\in[a,b]_\T$, are all Hermitian with $P^\rho(a)$ and $P(b)$ invertible, $\omega(t)>0$ for $t\in[a,b]_\T$, $R$ and $S$ in \eqref{qbc} are $2d\times 2d$ matrices with $\rank(R,S)=2d$, and the boundary conditions in \eqref{qbc} are self-adjoint. Let $r=\rank \left(R_1+\frac{1}{\mu_\rho(a)}S_1P^{\rho}(a),\frac{1}{\mu_\sigma(b)}S_2\right)$ and $m=d(N-2)+r$.
\begin{enumerate}
 \item The eigenvalue problem \eqref{introeq}, \eqref{qbc} has $m$ real eigenvalues
       \begin{equation}\label{sc41} 
       \lambda_1, \lambda_2,\cdots, \lambda_m,
       \end{equation}
       including multiplicities, and $m$ linearly independent eigenfunctions
       \begin{equation}\label{sc42} 
       x(\cdot, \lambda_1), x(\cdot, \lambda_2),\cdots, x(\cdot, \lambda_m),
       \end{equation}
       that form an orthonormal set, to wit
       \begin{equation}\label{sc43} 
       \langle x(\cdot, \lambda_j), x(\cdot, \lambda_k)\rangle=\sum_{t\in[a,b]_\T} 
       x^*(t,\lambda_k)\omega(t)x(t,\lambda_j)\mu_\sigma(t)=\delta_{jk}, \quad j,k\in\{1,2,\cdots,m\}.
       \end{equation}
 \item The eigenfunction basis for \eqref{introeq}, \eqref{qbc} consists of $m$ linearly independent eigenfunctions \eqref{sc42} and is complete for 
       admissible function space $L_{\omega}^2[\rho(a),\sigma(b)]$, to wit for each $x\in L_{\omega}^2[\rho(a),\sigma(b)]$, there exists a unique set 
       of scalars $\{c_j\}_{j=1}^m\subset\C$ such that, for $t\in[\rho(a),\sigma(b)]_\T$, we have
       \begin{equation}\label{sc44}
       x(t)=\sum_{j=1}^m c_jx(t,\lambda_j),
       \end{equation}
       where
       \begin{equation}\label{sc45}
       c_j=\langle x, x(\cdot, \lambda_j\rangle = \sum_{t\in[a,b]_\T}x^*(t,\lambda_j)\omega(t)x(t)\mu_\sigma(t), \quad j\in\{1,2,\cdots,m\},
       \end{equation}
       and the following Parseval equality holds, namely
       \begin{equation}\label{sc46}
       \langle x,x\rangle = \sum_{j=1}^m |c_j|^2.
       \end{equation}
 \item The dynamic operator $L$ has the spectral resolution
       \begin{equation}\label{sc47}
       Lx(t)=\sum_{j=1}^m \lambda_j \pi_jx(t) = \int_{-\infty}^{\infty} \lambda dE_{\lambda}x(t), \quad x\in  L_{\omega}^2[\rho(a),\sigma(b)], \quad 
       t\in[\rho(a),\sigma(b)]_\T,
       \end{equation}
       with $Lx^\rho(a)$ and $Lx^\sigma(b)$ given as in \eqref{sc316}, the projective operators are given by
       \begin{equation}\label{sc48}
        \pi_j x(t,\lambda_j)=\delta_{jk}x(t,\lambda_k), \quad j,k\in\{1,2,\cdots,m\}, \quad t\in[\rho(a),\sigma(b)]_\T,
       \end{equation}
       and the projective operator-valued function is given by
       \begin{equation}\label{sc49}
        E_{\lambda}=\begin{cases} \sum_{0<\lambda_j\le \lambda} \pi_j, &: \lambda\ge 0, \\ -\sum_{\lambda<\lambda_j\le 0}\pi_j, &: \lambda<0. \end{cases}
       \end{equation}
\end{enumerate}
\end{theorem}

\begin{theorem}[Dual Orthogonality]
Assume the $d\times d$ matrices $P(t)$ for $t\in[\rho(a),b]_\T$, $Q(t)$ and $\omega(t)$ for $t\in[a,b]_\T$, are all Hermitian with $P^\rho(a)$ and $P(b)$ invertible, $\omega(t)>0$ for $t\in[a,b]_\T$, $R$ and $S$ in \eqref{qbc} are $2d\times 2d$ matrices with $\rank(R,S)=2d$, and the boundary conditions in \eqref{qbc} are self-adjoint. Let $r=\rank \left(R_1+\frac{1}{\mu_\rho(a)}S_1P^{\rho}(a),\frac{1}{\mu_\sigma(b)}S_2\right)$ and $m=d(N-2)+r$. If the boundary conditions \eqref{qbc} are proper, that is $r=2d$, then $m=dN$ and the $dN$ eigenfunctions \eqref{sc42} in Theorem $\ref{scthm41}$ satisfy the dual orthogonality condition
\begin{equation}\label{sc410}
 \sum_{j=1}^{dN} x(t,\lambda_j)x^*(\tau,\lambda_j)=\delta(t,\tau)\omega^{-1}(\tau), \quad t,\tau\in[a,b]_\T, \quad \delta(t,\tau)=\begin{cases} 0 &:t\neq \tau, \\ 1 &: t=\tau. \end{cases}
\end{equation}
\end{theorem}

\begin{remark}
There are additional variational properties of eigenvalues and comparison results, including Rayleigh's principle and a minimax theorem, that carry over verbatim from the discrete case. We refer the interested reader to \cite[Section 5]{sc1}. One could also use the recent results by Shi and Wu \cite{sw} in the discrete case to improve our results above.
\end{remark}


\section{Further Extensions to Sturmian Theory}

For the discussion that follows, assume the time scale $\T$ is general, that is to say not necessarily isolated. In the classic continuous case $(\T=\R)$, scalar, vector, and matrix equations of the form
$$ -(px')'(t)+q(t)x(t)=0 $$
are embedded in a robust theory that extends to even-order Sturm-Liouville equations, linear Hamiltonian and symplectic systems; see, for example, Reid \cite{Reid3} and Kratz \cite{kratz}. It is common on general time scales to consider these same connections using the basic second-order scalar equation
\begin{equation}\label{ddeq}
 -(px^\Delta)^\Delta(t)+q(t)x^\sigma(t)=0. 
\end{equation}
Erbe and Hilger \cite{eh} began a Sturmian study of this equation, followed by Agarwal, Bohner, and Wong \cite{abw}. Recently Kong \cite{kong} extended some of these results. To the best of our knowledge, however, it is an open problem on how to extend \eqref{ddeq} to even-order self-adjoint Sturm-Liouville problems using only delta derivatives. 

As seen above there is a natural alternative theory that builds on the scalar equation
$$ -(px^\Delta)^\nabla(t)+q(t)x(t)=0. $$
After Atici and Guseinov introduced the nabla derivative \cite{ag} as a left-hand counterpart on time scales to the delta derivative, there has been an attempt to formulate a Sturmian theory using both delta and nabla derivatives. Messer \cite{messer} continued the work of \cite{ag} by studying second-order scalar self-adjoint equations; this work was extended to higher order equations, first by Guseinov \cite{gus}, and then by Anderson, Guseinov, and Hoffacker \cite{agh}; see also Davidson and Rynne \cite{dr}. A Reid roundabout theorem and an analysis of dominant and recessive solutions is given in \cite{ab,and} for the related matrix equation \eqref{saeq2}, given below.

A next step might be to extend the results of this paper to even-order Sturm-Liouville equations and linear Hamiltonian systems on arbitrary time scales, or more generally to symplectic systems on time scales. We illustrate this in one possible treatment not found in the literature. For example, we reconsider \eqref{saeq} rewritten in the related homogeneous second-order self-adjoint matrix dynamic equation in the form
\begin{equation}\label{saeq2}
 -\left(PX^\Delta\right)^\nabla(t)+Q(t)X(t)=0, \quad t\in[a,b]_\T.
\end{equation}
Let $P$ be invertible Hermitian, $Q$ ld-continuous Hermitian, and let $\D$ denote the set of all $n\times n$ matrix-valued functions $X$ defined on $\T$ such that $X^\Delta$ is continuous on $[\rho(a),b]_\T$ and $(PX^\Delta)^\nabla$ is left-dense continuous on $[a,b]_\T$. Then $X$ is a solution of \eqref{saeq2} on $[\rho(a),\sigma(b)]_\T$ provided $X\in\D$ and $X(t)$ satisfies the equation for all $t\in[a,b]_\T$.

\begin{lemma}
Let $X\in\D$. If we set
\begin{equation}\label{zs}
  Z:=\left(\begin{smallmatrix} X \\ PX^\Delta \end{smallmatrix}\right)  \quad\text{and}\quad
	S:=\left(\begin{smallmatrix} -\mu_\rho P^{\rho-1}Q & P^{\rho-1} \\ Q & 0 \end{smallmatrix}\right), 
\end{equation}
then $X$ solves \eqref{saeq2} on $[a,b]_\T$ if and only if $Z$ solves
\begin{equation}\label{me2}
	Z^\nabla = S(t)Z, \quad t\in[a,\sigma(b)]_\T.
\end{equation}
\end{lemma}

\begin{proof}
We give an idea of the proof. Let $X\in\D$ and $Z,S$ be given as in \eqref{zs}; multiplying out, we get 
\begin{eqnarray*}
 SZ &=& \left(\begin{smallmatrix} -\mu_\rho P^{\rho-1}Q & P^{\rho-1} \\ Q & 0 \end{smallmatrix}\right)\left(\begin{smallmatrix} X \\ PX^\Delta \end{smallmatrix}\right) 
 = \left(\begin{smallmatrix} -\mu_\rho P^{\rho-1}QX+P^{\rho-1}PX^\Delta \\ QX \end{smallmatrix}\right) \\
 &=& \left(\begin{smallmatrix} -\mu_\rho P^{\rho-1}(PX^\Delta)^\nabla+P^{\rho-1}PX^\Delta \\ QX \end{smallmatrix}\right) \\
 &=& \left(\begin{smallmatrix} -P^{\rho-1}((PX^\Delta)-(PX^\Delta)^\rho)+P^{\rho-1}PX^\Delta \\ QX \end{smallmatrix}\right) \\
 &=& \left(\begin{smallmatrix} P^{\rho-1}(PX^\Delta)^\rho \\ QX \end{smallmatrix}\right) = \left(\begin{smallmatrix} X^\nabla \\ QX \end{smallmatrix}\right) \\
 &=& \left(\begin{smallmatrix} X^\nabla \\ (PX^\Delta)^\nabla \end{smallmatrix}\right) = Z^\nabla,
\end{eqnarray*} 
where $(PX^\Delta)^\rho=P^\rho X^\nabla$ since $X\in\D$ implies $X^\Delta$ is continuous.
\end{proof}

Next we introduce a new type of Hamiltonian for the nabla differential operator, not previously discussed in the literature, although clearly in parallel with that available for the delta differential operator. In the discrete case see Ahlbrandt and Peterson \cite{ap}, and Shi \cite{shi1,shi2}, and see Bohner and Peterson \cite[Chapter 7]{bp1} for the delta operator on time scales.

\begin{definition}\label{d111}
A $2n\times 2n$-matrix-valued function $\mathcal{H}$ is Hamiltonian with respect to $\T$ provided
\begin{equation}\label{14}
  \mathcal{H}(t)\mbox{ is Hamiltonian and }
  I+\mu_\rho(t)\mathcal{H}(t)\mathcal{M}^*\mathcal{M}\mbox{ is invertible for all }t\in\T,
\end{equation}
where we have used the $2n\times 2n$-matrix $\mathcal{M}=\left(\begin{smallmatrix} 0 & I \\ 0 & 0 \end{smallmatrix}\right)$. For $\mathcal{H}$ satisfying \eqref{14}, the system
\begin{equation}\label{15}
  z^\nabla=\mathcal{H}(t)[\mathcal{M}^*\mathcal{M} z^\rho + \mathcal{M}\mathcal{M}^*z]
\end{equation}
is then called a {\em linear Hamiltonian dynamic nabla system}.
\end{definition}

\begin{remark}\label{r112}
Take $\mathcal{J}$ to be the $2n\times 2n$-matrix $\mathcal{J}=\left(\begin{smallmatrix} 0 & I \\ -I & 0 \end{smallmatrix}\right)$ and
$\mathcal{H}=\left(\begin{smallmatrix} \mathcal{A} & \mathcal{B} \\ \mathcal{C} & -\mathcal{A}^* \end{smallmatrix}\right)$ for some $n\times n$-matrix-valued functions $\mathcal{A}$, $\mathcal{B}$, $\mathcal{C}$ such that $\mathcal{B}$ and $\mathcal{C}$ are Hermitian. Then $\mathcal{H}^*\mathcal{J}+\mathcal{J}\mathcal{H}=0$, so that $\mathcal{H}$ is Hamiltonian, and 
$$ I+\mu_\rho\mathcal{H}\mathcal{M}^*\mathcal{M}= \left(\begin{smallmatrix} I & \mu_\rho \mathcal{B} \\ 0 & I-\mu_\rho \mathcal{A}^*\end{smallmatrix}\right). $$
Therefore, condition \eqref{14} may be rewritten as $\mathcal{H}=\left(\begin{smallmatrix} \mathcal{A} & \mathcal{B} \\ \mathcal{C} & -\mathcal{A}^*\end{smallmatrix}\right)$, where $I-\mu_\rho\mathcal{A}^*$ is invertible and $\mathcal{B}, \mathcal{C}$ are Hermitian. If we let $z=\left(\begin{smallmatrix} x \\ u \end{smallmatrix}\right)$, we can rewrite the system \eqref{15} as
\begin{equation}\label{nabham}
 x^\nabla = \mathcal{A}(t)x + \mathcal{B}(t)u^\rho,\quad
 u^\nabla = \mathcal{C}(t)x - \mathcal{A}^*(t)u^\rho. 
\end{equation}
Also note that equation \eqref{saeq2} is equivalent to a Hamiltonian system by taking $\mathcal{A}=0$, $\mathcal{B}=P^{\rho-1}$, $\mathcal{C}=Q$, $x=X$, and $u=PX^\Delta$. Finally, note that \eqref{nabham} is the form of the discrete Hamiltonian introduced by Ahlbrandt \cite{ahl}.
\end{remark}

\begin{theorem}\label{thm113}
If we set
\begin{equation}\label{17}
 \mathcal{S}=(I+\mu_\rho\mathcal{H}\mathcal{M}^*\mathcal{M})^{-1}\mathcal{H},
\end{equation}
then every Hamiltonian system \eqref{15} is also a symplectic nabla system $z^\nabla=\mathcal{S}(t)z$.
\end{theorem}

\begin{proof}
We assume first that $z$ solves \eqref{15}. Using the simple formula $z^\rho=z-\mu_\rho z^\nabla$ and observing that $\mathcal{M}^*\mathcal{M}+\mathcal{M}\mathcal{M}^*=I$, we obtain $(I+\mu_\rho\mathcal{H}\mathcal{M}^*\mathcal{M})z^\nabla=\mathcal{H} z$, 
so that \eqref{17} follows from \eqref{14}. To verify that $\mathcal{S}$ defined by \eqref{17} is symplectic with respect
to $\T$, videlicet satisfies condition $\mathcal{S}^*(t)\mathcal{J}+\mathcal{J}\mathcal{S}(t)=\mu_\rho(t)\mathcal{S}^*(t)\mathcal{J}\mathcal{S}(t)$ for all $t\in\T$,  note that
$$  \mathcal{S}=\left(\begin{smallmatrix} \mathcal{A}-\mu_\rho\mathcal{B}(I-\mu_\rho\mathcal{A}^*)^{-1}\mathcal{C} & \mathcal{B}(I-\mu_\rho\mathcal{A}^*)^{-1} \\ (I-\mu_\rho\mathcal{A}^*)^{-1}\mathcal{C} & -(I-\mu_\rho\mathcal{A}^*)^{-1}\mathcal{A}^*\end{smallmatrix}\right). $$
It is straightforward to check that $\mathcal{S}^*\mathcal{J}+\mathcal{J}\mathcal{S} =\mu_\rho\mathcal{S}^*\mathcal{J}\mathcal{S}$ holds on $\T$.
\end{proof}

As a final important illustration, consider even-order Sturm-Liouville dynamic equations; see a related discussion by Shi and Chen \cite{sc2} on higher-order discrete Sturm-Liouville problems. This is a new formulation not previously presented, but clearly in the spirit of the classical Sturm-Liouville analysis in the continuous case; see Kratz \cite[(6.2.1)]{kratz} and Shi and Chen \cite[p. 8]{sc2}.

\begin{example}
Consider the even-order Sturm-Liouville dynamic equation
\begin{eqnarray} \label{My}
My(t) &=& \sum\limits_{k=0}^n (-1)^{n-k}\left( p_{n-k}y^{\nabla^{n-k-1}\Delta} \right)^{\Delta^{n-k-1}\nabla}(t) \\ \nonumber
  &=& (-1)^n\left(p_ny^{\nabla^{n-1}\Delta}\right)^{\Delta^{n-1}\nabla}(t) + 
  \dots  - \left(p_{3}y^{\Delta^2\nabla}\right)^{\nabla^2\Delta}(t) \\ 
  & & + \left(p_{2}y^{\nabla\Delta}\right)^{\Delta\nabla}(t) - \left(p_{1}y^\Delta\right)^\nabla(t) 
      + p_0(t)y(t), \nonumber
\end{eqnarray}
which is formally self-adjoint \cite{agh}, $p_n\neq 0$. We will show \eqref{My} can be written in the form of \eqref{nabham}, where
\begin{gather*}
  \mathcal{A}=(a_{ij})_{1\leq i,j\leq n} \quad\text{with}\quad
   a_{ij}=\begin{cases}1: & \text{if } j=i+1,\; 1\leq i\leq n-1, \\
                       0: & \text{otherwise,} \end{cases}\\
  \mathcal{B}=\diag\left\{0,\dots,0,\frac{1}{p_n^\rho}\right\},\quad
  \mathcal{C}=\diag\left\{p_0,p_1^\rho,p_2^\rho,\ldots,p_{n-1}^\rho\right\}.
\end{gather*}
\end{example}
\noindent To do this, we introduce the pseudo-derivatives of the function $y$ given by
\begin{eqnarray}
  y^{[k]} &=& y^{\nabla^k}, \quad 0\le k\le n-1, \quad y^{[0]} = y^{\nabla^0} = y,  \label{yk} \\
  y^{[n]} &=& p_ny^{\nabla^{n-1}\Delta}, \nonumber \\
  y^{[n+k]} &=& p_{n-k}y^{\nabla^{n-k-1}\Delta} - \left(y^{[n+k-1]}\right)^\Delta \nonumber \\ 
   &=&\sum_{i=0}^k(-1)^{k-i}\left(p_{k-i}y^{\nabla^{n-i-1}\Delta}\right)^{\Delta^{k-i}}, \quad 1 \le k \le n-1, \nonumber \\
  y^{[2n]} &=& p_0y - \left(y^{[2n-1]}\right)^\nabla=My. \label{y2n}
\end{eqnarray}
From \eqref{yk} and \eqref{y2n} we have 
\begin{eqnarray*}
  && \left(y^{[k]}\right)^\nabla = y^{[k+1]}, \quad 0 \le k \le n-2,  \\
  && \left(y^{[n-1]}\right)^\nabla = y^{\nabla^n}=\frac{1}{p_n^\rho}\left(y^{[n]}\right)^\rho, \\
  && \left(y^{[n+k-1]}\right)^\nabla = p_{n-k}^\rho y^{[n-k]} - \left(y^{[n+k]}\right)^\rho, \quad 1 \le k\le n-1, \\
  && \left(y^{[2n-1]}\right)^\nabla = p_0 y - My. 
\end{eqnarray*}
For more details on the establishment of related (but different) types of formulas, please see \cite[Section 3]{agh}. Then using the substitution
$$ x = \left(\begin{smallmatrix} y^{[0]} \\ y^{[1]} \\ \vdots \\ y^{[n-1]} \end{smallmatrix}\right), \quad
   u = \left(\begin{smallmatrix} y^{[2n-1]} \\ y^{[2n-2]} \\ \vdots \\ y^{[n]} \end{smallmatrix}\right), $$
and the matrices $\mathcal{A}$, $\mathcal{B}$, and $\mathcal{C}$ above, we have that 
$$ x^\nabla = \mathcal{A}(t)x + \mathcal{B}(t)u^\rho, \quad u^\nabla = \mathcal{C}(t)x - \mathcal{A}^*(t)u^\rho, $$
and the example is complete.

Consider extending the discrete linear Hamiltonian analysis given by Shi \cite{shi1} to an eigenvalue problem for the Hamiltonian nabla system \eqref{nabham} on general time scales (including non-isolated time scales), namely
$$ x^\nabla = \mathcal{A}(t)x + \mathcal{B}(t)u^\rho, \quad u^\nabla = \left[\mathcal{C}(t)-\lambda\omega(t)\right]x - \mathcal{A}^*(t)u^\rho, \quad t\in[a,b]_\T $$
with boundary conditions
\begin{equation}\label{hambc}
 R\left(\begin{smallmatrix} -x^\rho(a) \\ x(b) \end{smallmatrix}\right)+S\left(\begin{smallmatrix} u^\rho(a) \\ u(b) \end{smallmatrix}\right)=0.
\end{equation}
Here $\mathcal{A}$, $\mathcal{B}$, $\mathcal{C}$ and $\omega$ are $d\times d$ matrices, $\mathcal{B}$ and $\mathcal{C}$ are Hermitian, $\omega>0$ is positive definite, $A^*$ denotes the complex conjugate transpose of $A$, and $R$ and $S$ are as in the first part of the paper, to wit $2d\times 2d$ matrices with $\rank(R,S)=2d$. Write 
$$ H(t) = \left(\begin{smallmatrix} -\mathcal{C}(t) & \mathcal{A}^*(t) \\ \mathcal{A}(t) & \mathcal{B}(t) \end{smallmatrix}\right), \quad 
   W(t) = \left(\begin{smallmatrix} \omega(t) & 0 \\ 0 & 0 \end{smallmatrix}\right), \quad J = \left(\begin{smallmatrix} 0 & -I \\ I & 0 \end{smallmatrix}\right). $$
Then $H$ and $W$ are $2d\times 2d$ Hermitian matrices with $W(t)\ge 0$ for $t\in[a,b]_\T$, and the system \eqref{nabham}, \eqref{hambc} can be recast in the form
\begin{equation}\label{shieq14}
 J \left(\begin{smallmatrix} x^\nabla(t) \\ u^\nabla(t) \end{smallmatrix}\right) = \left[H(t)+\lambda W(t)\right] \left(\begin{smallmatrix} x(t)  \\ u^\rho(t) \end{smallmatrix}\right).
\end{equation}
Define the natural dynamic nabla differential operator for \eqref{nabham}, \eqref{hambc} via
\begin{equation}\label{shieq15}
 \ell(x,u)(t) = J \left(\begin{smallmatrix} x^\nabla(t) \\ u^\nabla(t) \end{smallmatrix}\right) - H(t) \left(\begin{smallmatrix} x(t)  \\ u^\rho(t) \end{smallmatrix}\right).
\end{equation}
Clearly this includes \eqref{introeq}, \eqref{qbc} as a special case, if we replace $b$ by $\sigma(b)$ as the right-hand endpoint to accommodate $u=Px^\Delta$. As before we introduce the linear space
$$ \ell[\rho(a),b]:=\left\{\left\{(x(t),u(t))\right\}_{t\in[\rho(a),b]_\T}:x(t),u(t)\in\C^d, t\in[\rho(a),b])_\T\right\}, $$
and the weighted inner product
$$ \langle x,y \rangle = \int_{\rho(a)}^{b} y^*(t)\omega(t)x(t)\nabla t, \quad (x,u),(y,v)\in\ell[\rho(a),b]. $$
Then we have the following key result.

\begin{theorem}[Lagrange Identity]
For all $(x,u),(y,v)\in\ell[\rho(a),b]$, we have
$$ \int_{\rho(a)}^{b}\left\{\left(y^*,v^{\rho*}\right)\ell(x,u) - \ell(y,v)^* \left(\begin{smallmatrix} x \\ u^\rho \end{smallmatrix}\right) \right\}(t)\nabla t = \left(y^*(t),v^*(t)\right)J\left(\begin{smallmatrix} x(t) \\ u(t) \end{smallmatrix}\right)\Big|_{\rho(a)}^{b}. $$
\end{theorem}

\begin{proof}
Suppressing the variable $t$, we have
$$ (y^*,v^{\rho*})\ell(x,u) = -y^*u^\nabla+y^*\mathcal{C}x-y^*\mathcal{A}^*u^\rho+v^{\rho*}x^\nabla-v^{\rho*}\mathcal{A}x-v^{\rho*}\mathcal{B}u^\rho, $$
$$ \ell(y,v)^* \left(\begin{smallmatrix} x \\ u^\rho \end{smallmatrix}\right) = -v^{\nabla*}x+y^*\mathcal{C}x-v^{\rho*}\mathcal{A}x+y^{\nabla*}u^\rho-y^*\mathcal{A}^*u^\rho-v^{\rho*}\mathcal{B}u^\rho,  $$
so that when we subtract the second from the first, we obtain
\begin{eqnarray*}
 (y^*,v^{\rho*})\ell(x,u) - \ell(y,v)^* \left(\begin{smallmatrix} x \\ u^\rho \end{smallmatrix}\right) 
 &=& -y^*u^\nabla + v^{\rho*}x^\nabla + v^{\nabla*}x -y^{\nabla*}u^\rho \\
 &=& (v^*x)^\nabla-(y^*u)^\nabla.
\end{eqnarray*}
The result follows from the fundamental theorem of calculus.
\end{proof}

A spectral theory for dynamic linear Hamiltonian nabla systems should now be possible; the interested reader is encouraged to see \cite{shi1} for the details in the discrete case.


\end{document}